\documentclass[reqno]{amsart}
\usepackage{amssymb,amsthm}
\usepackage[dvips]{graphicx}
\newcommand{\R}{{\mathbb R}}

\newcommand{\C}{{\mathbb C}}
\newcommand{\Z}{{\mathbb Z}}
\newcommand{\CS}{\mathcal S}
\newcommand{\cal}{\mathcal}
\newcommand{\WT}[1]{\widetilde{#1}}
\newtheorem{theorem}{Theorem}[section]
\newtheorem{corollary}[theorem]{Corollary}
\newtheorem{lemma}[theorem]{Lemma}

\newtheorem{proposition-definition}[theorem]{Proposition - Definition}
\newtheorem{definition}[theorem]{Definition}
\newtheorem{remark}[theorem]{Remark}

\newtheorem{hypothesis}[theorem]{Induction hypothesis}
\numberwithin{equation}{section}

\begin{document}

\title[Gradient-like vector fields on a complex analytic variety]{Gradient-like vector fields \\on a complex analytic variety}
\author{Cheol-Hyun Cho and Giovanni Marelli}

\begin{abstract}
Given a complex analytic function $f$ on a Whitney stratified complex analytic variety of complex dimension $n$, whose real part $Re(f)$ is Morse, we prove the existence of a 
stratified gradient-like vector field for $Re(f)$ such that the unstable set of a critical point $p$ on a stratum $S$ of complex dimension $s$ has real dimension $m(p)+n-s$ as was conjectured by Goresky and MacPherson.
\end{abstract}

\thanks{First author was supported by the Basic research fund 2009-0074356 funded by the Korean government}
\address{Department of Mathematical Sciences, Seoul National University,
San 56-1, Shinrimdong, Gwanak-gu, Seoul, South Korea,  chocheol@snu.ac.kr \\
KIAS,  Hoegiro 87, Dongdaemun-gu, Seoul 130-722, South Korea, marelli@kias.re.kr}
\maketitle
\section{Introduction}
Morse theory has been a very powerful tool to study the topology and geometry of manifolds.
An analogue of such theory for singular stratified spaces has been developed by
Lazzeri \cite{La}, Pignoni \cite{Pi}, Goresky-MacPherson \cite{GMP} and by many other people. For a good review of
the development of stratified Morse theory, we refer readers to the article of Massey \cite{Ma1}. Goresky and MacPherson
proved the main theorem of stratified Morse theory that local Morse data is given as the product of
the tangential and normal Morse data in \cite{GMP} and  have given beautiful applications of stratified Morse theory to intersection cohomology and to the topology of complement of affine subspaces arrangements.

Many examples of singular stratified spaces are provided by complex analytic varieties, in which case the stratified Morse theory behaves much better. For example, Goresky and MacPherson showed that the normal Morse data of a point
in the stratum is determined by the complex link up to homotopy equivalence.
Also in \cite{GMP} they showed that a critical point of a Morse function of a (stratified) complex analytic variety
can be given an index, in the sense that the intersection homology of the Morse data is shown to be non-vanishing 
only in one index (Theorem I.1.6 \cite{GMP}). In the case of smooth manifolds, the Morse index of a critical point is given as the
dimension of the unstable manifold of the negative gradient flow.  But in the (stratified) case of complex analytic varieties, it has not been known whether the index 
given from the intersection homology can be interpreted as the dimension of the unstable (stratified) manifold or not.

In fact, Goresky and MacPherson conjectured the existence of a gradient-like vector field whose dimension of unstable (stratified) manifold is related to the above index (\cite{GMP} p 213). The goal of this paper is to provide the construction of this conjectural gradient-like vector field
for Morse functions obtained from complex analytic functions.

\begin{theorem}\label{thm:main}
Let $X$ denote a purely $n$-dimensional reduced complex analytic variety, on which we fix a Whitney stratification $\mathcal S$.
Let $f^c:X\rightarrow\C$ be a complex analytic function, and denote by $f := Re(f^c)$ its real part. Suppose that $f$ is a Morse function on $X$ and let $p$ be a critical point of $f$. Assume that
$\{p\}$ is a stratum of $\cal S$.
Then, there exists a stratified weakly controlled gradient-like vector field $V$ for $f$ with continuous flow and whose unstable and stable set $W^u(p)$ and $W^s(p)$ at $p$ satisfy for every $S \in \cal S$, 
\begin{eqnarray}\label{dimest1}
\dim_{\R} (W^u(p) \cap S) &\leq& \dim_{\C} S  \nonumber \\
\dim_{\R} (W^s(p) \cap S) &\leq& \dim_{\C} S 
\end{eqnarray}
\end{theorem}

For a general critical point $p$ of the Morse function $f$, the point $p$ itself does not form a stratum but
belongs to some $S_0 \in \cal S$. In this case, by restricting to a normal slice at $p$, and by applying the
main theorem \ref{thm:main}, we obtain the following corollary, which justifies the definition of Morse index in
the complex analytic case.
\begin{corollary}\label{cor:main}
Let $X, \mathcal S, f^c,f$ be as above.
Then, there exists a stratified  weakly controlled gradient-like vector field $V$ for $f$ with continuous flow 
such that for any critical point $p$  of $f$, its unstable and stable set $W^u(p)$ and $W^s(p)$ at $p$ satisfy
\begin{eqnarray*}
\dim_{\R} W^u(p) &=&m_S(p)+n-s \\
\dim_{\R} W^s(p) &=& n+s-m_S(p)
\end{eqnarray*}
where $m_S(p)$ denotes the Morse index of $f_{|S}$ at $p$, $S$ is the stratum containing $p$, and $s=\dim_{\C}S$.
\end{corollary}
\begin{remark}
The dimension of the unstable set $m_S(p)+n-s$ equals the Morse index $m(p)$ defined in 
the definition \ref{def:morind}(\cite{GMP}). 
\end{remark}

The scheme of the proof of the main theorem is rather simple. Namely, we exploit  the beautiful construction, recalled in section 5, of the vanishing polyhedron performed by L\^e in \cite{L2} and
modify it to construct a field on $X$ with the prescribed unstable set and projecting via $f^c$ onto horizontal lines in the half disc $D^- := D^2 \cap \{z|Re(z) < 0 \}$.
We consider then
a similar construction over $D^+$ and glue them along the imaginary axis. As the constructed field may be trivial
over $(f^c)^{-1}(0)$, we modify the obtained flow by 
combining with another contribution, built by using the submersion similar to $f:X \setminus \{p\} \to \R$, so that
the resulting vector field satisfies the desired properties. 

There are related recent works by Misha Grinberg and Ursula Ludwig on this conjecture and we explain them
briefly in section 3.
 
We hope this work to be a first step to apply modern techniques of Morse theory or in general Floer theory
to complex analytic spaces. As a next step we hope to construct in the near future, the Morse-Witten-Smale complex of intersection cohomology using the construction of this paper. 

Here is the organization of this paper. In the next section, we recall basic notions in the theory of stratified spaces.
In section 3, we recall previous results related to the conjecture. In section 4 and 5, we explain polar curve and
L\^e's construction. In the last section we prove the main theorem.

We would like to thank L\^e D\~ung Tr\'ang, and Misha Grinberg for helpful correspondences and specially
for Misha Grinberg, who pointed out some mistakes in the original version. After this work was posted on the arXiv, Grinberg has proved similar results but with a different method recently in \cite{Gr2}.

\section{Preliminaries}
Let M be a smooth manifold and $X\subset M$ a closed subset endowed with a Whitney stratification.
We refer readers to \cite{P} for the definition of Whitney stratification and more details concerning stratified spaces and
vector fields on them.
\subsection{Stratified Morse theory}
First, we recall the definition of Morse function on $X$.
\begin{definition}\label{def:mor}
A Morse function $f:X\rightarrow\R$ is the restriction of a smooth function $f:M\rightarrow\R$ such that\\ 
(a) $f_{|X}$ is proper and its critical values are distinct\\ 
(b) for each stratum $S_i$ the critical points of $f|_{S_i}:S_i\rightarrow\R$ are non-degenerate\\ 
(c) for every such critical point $p\in S_i$ and for every generalized tangent space $Q$ at $p$, with $Q\neq T_pS_i$, we have $df(p)(Q)\neq0$.
\end{definition}

\begin{definition}
A gradient-like vector field for a Morse function $f$ on a stratified space $X$ is a vector field $V$ such that its restriction $V|_S$ to each stratum $S$ is a gradient-like vector field for $f|_S$, that is:\\
(a) for each critical point $p$ of $f|_S$ there is a neighbourhood $U_S(p)$ of $p$ in $S$ and a chart $\varphi:U_S(p)\rightarrow {\cal U}$ in $C^s$ where $\varphi_*(V)$ is in standard form 
\begin{displaymath}
-\sum_{i=1}^{m_S(p)} x_i\frac{\partial}{\partial x_i}+\sum^s_{i=m_S(p)+1} x_i\frac{\partial}{\partial x_i}
\end{displaymath}
where $s$ is the dimension of $S$ and $m_S(p)$ is the Morse index of $f_{|S}$ at $p$\\
(b) $V(x)(f|_S)>0$ for $x\notin\cup_p U_S(p)$
\end{definition}

In the complex analytic case, the following definition of index has been introduced in \cite{GMP}.
\begin{definition}\label{def:morind}
Let $f:X\rightarrow\R$ be a Morse function on a purely $n$-dimensional Whitney stratified space $X$ and let $p$ be a critical point of $f$ belonging to a stratum $S$ of dimension $s$; we define the Morse index of $p$ as
$$m(p)=m_S(p)+n-s$$
where $m_S(p)$ is the standard Morse index of $f_{|S}$ at $p$.
\end{definition}
We recall that Goresky and MacPherson (see again \cite{GMP} and also \cite{GMP1})has related the vanishing of the intersection homology of a Morse pair to the index of a critical point, extending to the stratified case the analogous classical theorem for Morse theory on smooth manifolds, after replacing singular homology with intersection homology. 
\begin{theorem}\cite{GMP1}
For a proper Morse function $f:X\rightarrow\R$ on a Whitney stratified complex analytic variety $X$ with a critical point $p$ with critical value $v=f(p)$, the intersection homology $IH_i(X_{\leq v+\epsilon},X_{\leq v-\epsilon})$ of Morse data at $p$ vanishes for all $i\neq m(p)$.
\end{theorem}

\subsection{Control systems}
We recall briefly the definitions of control data, controlled vector field and controlled lift.
(See \cite{Mat} or \cite{P} for more details.)
Let $X$ be a real $C^\infty$ manifold with a Whitney stratification $\CS$.
\begin{definition}
Let $S \in \CS$ be a stratum, and let $T_S$ be an open neighborhood of $S$ in $X$. A tubular projection
$\Pi_S:T_S \to S$ is a smooth submersion which is a retraction, and satisfies $\Pi_S|_S = id|_S$.
\end{definition}
\begin{definition}
Control data on $(X,\CS)$ is a collection $\{T_S,\Pi_S,\rho_S\}_{S \in \CS}$, where
$\Pi_S : T_S \to S$ is a tubular projection, and $\rho_S:T_S \to \R_{\geq 0}$ is
continuous function with $\rho^{-1}(0)=S$ with the following properties:
\begin{enumerate}
\item for each pair of strata $(S,R)$ with $T_S \cap R \neq \emptyset$ we have $S \subset \overline{R}$ (that is, $S<R$);
\item for each pair of strata $(S,R)$ with $S<R$, the map $(\Pi_S,\rho_S):T_S \cap R \to S \times \R_{\geq 0}$ is smooth and submersive;
\item for any two strata $S,R$ with $S<R$, the following compatibility conditions
hold
$$ \Pi_S \circ \Pi_R(x) = \Pi_S(x),\;\; \rho_S \circ \Pi_R(x) = \rho_S(x)$$
for any $x\in T_S\cap T_R$ with $\Pi_R(x)\in T_S$.
\end{enumerate}
\end{definition}
Two control data $\{T_S,\Pi_S,\rho_S\}_{S \in \CS}$ and $\{T'_S,\Pi'_S,\rho'_S\}_{S \in \CS}$ are equivalent over $S$ if there exists a neighbourhood $\tilde{T}_S\subset T_S\cap T'_S$ such that
$$\Pi_{S|\tilde{T}_S}=\Pi'_{S|\tilde{T}_S}$$
$$\rho_{S|\tilde{T}_S}=\rho'_{S|\tilde{T}_S}$$
We will say that $(X,\CS)$ endowed with an equivalence class of control data $\{T_S,\Pi_S,\rho_S\}_{S \in \CS}$ is a controlled space.
\begin{definition}
Control data $\{T_S,\Pi_S,\rho_S\}_{S \in \CS}$ on $(X,\CS)$ are said to be compatible with a stratified map $h:X\rightarrow N$, where $N$ is a manifold with its trivial stratification, if for every stratum $S$ and all $x\in T_S$
$$h\circ\Pi_S(x)=h(x).$$
\end{definition}
\begin{theorem}
For every Whitney stratified space $X$ (and every smooth stratified submersion $h:X\rightarrow N$ to a manifold $N$) there exist control data (compatible with $h$).
\end{theorem}

\begin{definition}
A vector field $V:X\rightarrow TX$ over a controlled space $X$ is said to be controlled if there exist control data $\{T_S,\Pi_S,\rho_S\}_{S \in \CS}$ of $X$ such that for every $S<R$
$$d\Pi_{S_{|T_S\cap R}}\circ V_{|T_S\cap R}=V_{|S}\circ\Pi_{S_{|T_S\cap R}}$$
$$d\rho_{S_{|T_S\cap R}}\circ V_{|T_S\cap R}=0$$
\end{definition}
\begin{theorem}
Given $X$, $N$ and $h$ as above, there exists for every smooth vector field $W:N\rightarrow TN$ a controlled vector field $V:X\rightarrow TX$ such that
$$dh\circ V=W\circ h$$
\end{theorem}
As a special case we have:
\begin{corollary}
If $N=S\subset X$ and $V$ is a vector field on $S$, then it can be lifted to a controlled vector field on $T_S$.
\end{corollary}

We will also use a slightly different definition of controlled conditions, which we explain now:
\begin{definition}
\label{control}
A vector field $V:X\rightarrow TX$ over a controlled space $X$ is said to be weakly controlled if there exist control data $\{T_S,\Pi_S,\rho_S\}_{S \in \CS}$ of $X$ such that for every $S<R$
 $$d\Pi_{S_{|T_S\cap R}}\circ V_{|T_S\cap R}=V_{|S}\circ\Pi_{S_{|T_S\cap R}}$$
$$|d\rho_{S_{|T_S\cap R}}\circ V_{|T_S\cap R}|\leq A\rho_{S_{|T_S\cap R}}(x)$$ 
for some positive constant $A$
\end{definition}
(Weak) control conditions are interesting because they ensure that a vector field which satisfies them has a continuous flow and its gradient lines do not leave a stratum in finite time (see \cite{P}). However we need weak controllability to allow gradient lines to approach a point on a stratum of smaller dimension in an infinite time, which would not be possible with the standard controllability.

We also recall the weakly Lipschitz condition introduced by Verdier in \cite{V}
(In \cite{P}, it is called Verdier condition).
\begin{definition}
Let $A$ be a subanalytic set of an Euclidean vector space $V$ and $\{S_i\}$ a Whitney stratification of $A$: we say that a function $f:A\rightarrow\R$ is weakly Lipschitz if for every stratum $S_i$, $f_{|S_i}$ is smooth and for every $x\in S_i$ there exists a neighbourhood $U$ of $x$ in $V$ and a constant $C$ such that for every $x'\in U\cap S_i$ and for every $y\in U\cap A$
$$|f(x')-f(y)|\leq C\|x-y\|$$
\end{definition}

A weakly Lipschitz function is continuous but in general is not Lipschitz. The above definition can be extended to maps between analytic varieties by requiring the condition to hold locally and to vector-valued functions by imposing it to each component.

\begin{definition}
If $X$ is an analytic variety, $A$ a subanalytic set of $X$ endowed with a Whitney stratification $\cal S$, we say that a stratified vector field $V$ on $A$ is weakly Lipschitz if for every local immersion $\varphi$ of $X$ into a smooth manifold $M$, the vector field induced by $V$ on $\varphi^\ast (TM|_A)$ is a weakly Lipschitz section.
\end{definition}
We recall the theorem of Verdier (\cite{V}) which shows the usefulness of weakly Lipschitz condition.
\begin{theorem}
A stratified weakly Lipschitz vector field on a closed subanalytic set admits a flow
which is stratum preserving and is also weakly Lipschitz.
\end{theorem}

One of the main tools we use in this paper is the well-known Thom-Mather isotopy lemma:

\begin{theorem}
Let $h:X\rightarrow N$ be a proper controlled submersion (that is, $h$ is proper, controlled and its restriction to each stratum is submersive), then there exists a covering of $N$ by open subsets $U$ such that for each $U$ there is a stratified space $Y$ and an isomorphism of stratified spaces $\psi:Y\times U\rightarrow h^{-1}(U)$ such that $h_{|h^{-1}(U)}\circ \psi(y,x)=x$. We say that $h$ is locally trivial.
\end{theorem}

\section{Previous results}
We explain  previous results of Ludwig \cite{L} and Grinberg \cite{Gr} and their relations to this work.
\subsection{Ludwig's Morse-Witten-Smale complex}
First, Ludwig has constructed a version of  Morse-Smale-Witten homology on a stratified space  whose homology is isomorphic to singular homology,
{\it not} intersection homology.
\begin{theorem}\cite{L}
Let $X$ be a compact abstract stratified space and $(f,g)$ a stratified Morse pair satisfying the Morse-Smale condition. Then there is an isomorphism
$$H_{\ast}(f,g;\Z_2)\cong H_{sing}(X;\Z_2)$$
\end{theorem}
The Morse-Smale-Witten homology above is defined by a ``Morse pair'' which needs a few explanations:
on a compact abstract stratified space, defined as a stratified space admitting a controlled system (observe that, as a consequence, these spaces, of which Whitney stratified spaces are examples, possess a locally cone-like structure) she first considers a somewhat general class of ``Morse" functions.
More precisely, a point $p$ in a stratum $S$ is critical for a function $f$ if the restriction $f_{|S}$ has a critical point at $p$, and
it is a non-degenerate critical point if it is non-degenerate both in the tangential and normal direction.

She considers stratified metrics $g$ which are in some sense compatible with the given control systems, and whose existence is shown, to construct the stratified gradient vector field $\nabla_gf$ of a function $f$. She defines a critical point to be non-degenerate if, not only it is non-degenerate for $f_{|S}$ but also if there exists a control system with respect to which $\nabla_gf$ is a radial extension of $\nabla_gf_{|S}$. By radial extension of a vector field $V_S$ on $S$ it is meant a vector field $V$ extending $V_S$ to $T_S$ as sum of a controlled lift of $V_S$ to $T_S$, a stratified bounded vector field on $T_S$ with factor $\rho_S^2$ and a controlled lift by $\rho_S$ of the vector field $-t\partial/\partial t$ on $\R$. 

Note that due to the part  $-t\partial/\partial t$, non-degenerate stratified gradient vector field $\nabla_gf$ 
are forced to have a  flow to go only from larger to smaller strata. (This asymmetry seems to be related to the non-existence of
Poincar\'e duality on singular homology)

Now a stratified Morse pair is a pair $(f,g)$, where $f$ has no degenerate critical points and $\nabla_gf$ is a controlled vector field with respect to some control system. In \cite{L} the existence of such stratified Morse pair is proved.
As the flows go only from larger to smaller strata, the unstable set is a smooth submanifold contained in the stratum $S$ to which $p$ belongs (whose dimension so is well-defined and used as Morse index of the point $p$), while the stable set is a stratified space. With some transversality assumptions as in the smooth case, a Morse-Smale-Witten homology is constructed and showed to be isomorphic to singular homology.

However a generic Morse function does not define a Morse pair: consider a stratified space $X$ which is given by two 2-spheres glued at a single point, say $p$, which we regard as a singular stratum. Then, a function $f$ on $M$ gives rise to a Morse pair only if $f$ has a local minimum at the singular stratum $p$.
And in the case $f$ gives rise to a Morse pair, $-f$ does not.

In our case, instead we fix any Morse function and construct a gradient-like vector field with respect to this function. Also, we expect that such an approach, in the complex analytic case, will eventually provide a Morse-Smale-Witten complex for intersection homology (instead of singular homology).

\subsection{Grinberg's proof of the conjecture (\ref{thm:main}) up to a fuzz}
Now, we explain the result of Grinberg \cite{Gr}. There the main purpose is the construction of a self-indexing Morse function for a complex stratified space, which was proved by means of the following theorem.
\begin{theorem}\cite{Gr}
Let $X=\C^n$, ${\cal S}$ an algebraic Whitney stratification of $X$, $p\in X$ and $\Delta$ an open subset of the set of non-degenerate covectors at $p$ (that is, those covectors which does not annihilate any generalized tangent plane at $p$). Then there exists $f\in\Delta$, which may be considered as a linear function $f:X\rightarrow\R$, a closed ball $B$ around $p$ and a closed real semi-algebraic set $K\subset B$ such that
\begin{enumerate}
\item $\dim_\R K\cap S\leq \dim_\C S$ for every $S\in{\cal S}$;
\item $f^{-1}(0)\cap K=\{p\}$;
\item for every open $U\supset K$ there exists an ${\cal S}$-preserving $\nabla f$-like vector field $V$ on $B$ with stable and unstable sets contained in $U$ (here an ${\cal S}$-preserving gradient-like vector field is a controlled gradient-like vector field with respect to some control system).
\end{enumerate}
\end{theorem}
In other words, the theorem proves that the stable and unstable sets of the covector $f$ can be made sufficiently close to a subset $K$ having the expected dimension as in the conjecture (Theorem \ref{thm:main}).

But the theorem does not provide the actual dimension of the stable and unstable sets. In other words, as Grinberg himself states, it solves the conjecture ``up to a fuzz''. However this result was enough to prove the existence of a self-indexing Morse function:
\begin{theorem}\cite{Gr}
Every proper, non-singular, Whitney stratified complex analytic variety admits a self-indexing Morse function.
\end{theorem}

\section{The polar curve and the Cerf diagram}
For the rest of this paper, we denote by $X$ a purely $n$-dimensional reduced complex analytic subvariety of some complex analytic manifold $M$ (unless specified otherwise), on which we fix a Whitney stratification $\mathcal S$ (whose existence was proved by Whitney himself in \cite{W}).
Let $f^c:X\rightarrow\C$ be a complex analytic function and suppose $f:=Re(f^c)$  is a Morse function and $p$ a critical point of $f$ and also assume that $f(p) =0$ for simplicity. We now recall the notion of polar curve, which provides an essential and convenient tool to carry out inductive arguments for complex analytic varieties.

The following result, proved in different generality, by Hamm and L\^e in \cite{LH} and by L\^e
in \cite{L0} (see also chapter 7 of the book \cite{Mih}, \cite{Ma} or \cite{Ma2} Theorem 1.1):
\begin{theorem}\label{h1}
There exists an open dense Zariski subset $\Omega$ in the space of hyperplanes through $p$ such that for every $H\in\Omega$ we have that:\\
(1) there is an open neighborhood $U\subset X$ such that $H$ is transversal in $U$ to every stratum $S_i$ of $\mathcal S$ with $p\in\overline{S_i}$ except maybe $\{p\}$ if it is a stratum;\\
(2) chosen a linear form $l:X\rightarrow\C$ defining $H$, consider the map
$$\phi:X\rightarrow\C^2$$
$$\phi=(l,f^c)$$ 
then  for any stratum $S_i$ with $p\in\overline{S_i}$ the set 
$$\Gamma_{\overline{S_i}}:=\overline{(C_i\setminus\{Crit(f^c)\})}\cap U$$ 
where $C_i$ is the set of critical points of the restriction of $\phi$ to the smooth part of $\overline{S_i}$, is either empty or a reduced curve. Furthermore, $\Gamma_{\overline{S_i}}$ properly intersects
$(f^c)^{-1}(0)$ at $p$, i.e. $p$ is an isolated point in $\Gamma_{\overline{S_i}} \cap (f^c)^{-1}(0)$.
\end{theorem}
This theorem leads to the defintions of so called polar curve and Cerf diagram:
\begin{definition} 
We fix $H \in \Omega$ given in the Theorem \ref{h1}.
We call 
$$\Gamma=\cup_i(\Gamma_{\overline{S_i}}\cap U)$$
the polar curve of $f^c$ at $p$ relatively to the stratification $\mathcal S$. If $\Gamma$ is non-empty, for a suitable choice of $U$, $\phi(\Gamma)=\Delta = \cup_{\alpha \in A} \Delta_\alpha$ is an analytic curve in an open set of $\C^2$, which is called the Cerf's diagram of $\phi$ at $p$ relatively to the stratification $\mathcal S$.
\end{definition}
Now we discuss the relation between a Morse function and hyperplane sections.
\begin{lemma}\label{lem:restmor}
Let $f:X\to \R$ be a Morse function on a complex analytic subspace $X$ of some complex analytic manifold $M$ as above.
Then, there exists a dense subset $\Omega'$ in the space of hyperplanes through $p$ such that, 
for all $H\in \Omega'$, the function $f:X\cap H \to \R$ is also a Morse function.
\end{lemma}
\begin{proof}
Suppose not. Then it is easy to see that condition $(c)$ in Definition \ref{def:mor} fails to hold for $f:X\cap H \to \R$ for a
generic choice of $H$. First, suppose that the union $U$ of all generalized tangent spaces at $p$ is a subspace
of a complex subspace $V \subset T_pM$, where $dim_\C (V) < dim_\C(T_pM)$. Then, for any generalized tangent
plane $Q$, $Q\cap H = \{0\}$ for generic $H$, hence condition $(c)$ trivially holds.
Hence, we may suppose that $U  = T_pM$ as $U$ is a complex vector space.

But the failure of condition $(c)$ implies the existence of a nonzero vector $v_H \in Q\cap H$ such that
$df(p)(v_H)=0$. As we can choose a dense family of hyperplanes $H$, we can find a dense family of vectors $v_H$ with
such property. As $df(p)$ is a linear function, this implies the vanishing of $df(p)$ on $Q$, providing a contradiction.
\end{proof}
In this paper, we will choose hyperplanes in $\Omega'\cap \Omega$.
Polar curve has provided very useful tool to study singular complex analytic spaces. 

The famous fibration theorem of Milnor has been generalized by L\^e in \cite{L02}. It is useful to consider neighborhoods
which are not balls but polydiscs, as done by L\^e in the following topological preparation theorem.

\begin{theorem}\cite{L2}
Let $l \in \Omega$. For simplicity, we suppose $l(p) =f^c(p) =0$. For all $\epsilon$,
$1 \gg \epsilon >0$ and for all $\eta_1,\eta_2$, $\epsilon \gg \eta_1>0$, $\eta_1 \gg \eta_2>0$,
the morphism $\phi$ induces a mapping of $X_{\epsilon,\eta_1,\eta_2} =
B_\epsilon(p) \cap \phi^{-1}(D_{\eta_1} \times D_{\eta_2})\cap X$ on $D_{\eta_1} \times D_{\eta_2}$
which induces an topological fibration of $X_{\epsilon,\eta_1,\eta_2} \setminus \phi^{-1}(\cup_{\alpha \in A} \Delta_{\alpha})$ on $D_{\eta_1} \times D_{\eta_2} \setminus (\cup_{\alpha \in A} \Delta_{\alpha})$, moreover
the projection onto $D_{\eta_2}$ composed with $\phi$ induces a map on $X_{\epsilon,\eta_1,\eta_2}$ onto $D_{\eta_2}$
which induces a topological fibration above $D^*_{\eta_2}$ isomorphic to that  of local fibration theorem.
\end{theorem}

We also recall that a complex analytic function $g^c:X\subset M\rightarrow\C$ has an isolated singularity at $x$ if for all complex analytic extension
 $\tilde{g^c}$ of $g$ in an open neighborhood $U$ of $x$ in $M$,  the intersection of the image of the section $d\tilde{g^c}$ and
the conormal space $\cup_{i} T_{\overline{S_i}}^*U$ in $(x,d\tilde{g^c}(x))$ is an isolated point. The critical
points of a Morse function are isolated singularites and we refer to \cite{L2} or \cite{L3} for more details.

\section{L\^e's inductive construction of vanishing polyhedron}
\label{leconstr}
For the main construction of this paper, we recall briefly parts of L\^e's constructions in \cite{L1} and \cite{L2}. The construction in this section is entirely due to L\^e and we refer readers to \cite{L2} for its full details.

Consider the map $\phi:X \to \C^2$ from Theorem \ref{h1}. L\^e has shown how to define a polyhedron 
which contains essential information of each fiber $f^{-1}(t)$ (see figure \ref{fig1}), by working explicitly with polar curves.

Fix $\epsilon, \eta$ so that the fibration theorem of \cite{L02} holds  and that $f^c$ has rank one in such a neighborhood. We denote (for $t \in D_\eta \setminus \{0\}$)
$$X_{\epsilon,\eta} = B_\epsilon \cap (f^c)^{-1}(D_\eta), \;\; X_{\epsilon,\eta}(t) = B_\epsilon \cap (f^c)^{-1}(t),$$
  $$\partial X_{\epsilon,\eta} = \big( \partial B_\epsilon \cap (f^c)^{-1}(D_\eta) \big) \cup \big( B_\epsilon \cap (f^c)^{-1}(S^1_\eta) \big),\;\; \partial X_{\epsilon,\eta}(t) = \partial B_\epsilon \cap (f^c)^{-1}(t).$$

\begin{theorem}\cite{L2}
Let $f^c:X\rightarrow\C$ be a complex analytic function on a purely $n$-dimensional reduced complex analytic variety. Suppose $X$ is closed in an open subset of $\C^N$ for some $N$ and endowed with a Whitney stratification ${\cal S}$. Suppose $f^c$ has an isolated singularity at $p\in X$ and for simplicity that $f^c(p)=0$. 

Then for $\epsilon$ and $\eta$ such that $1\gg\epsilon\gg\eta>0$ and for all $t\in  D_\eta \setminus \{0\}$, there exists in $X_{\epsilon,\eta}(t)$ a polyhedron $P_t$ of real dimension $n-1$ compatible with ${\cal S}$ (that is, the interior of each simplex is contained in some stratum of ${\cal S}$) and a continuous stratified simplicial map $\psi_t:\partial X_{\epsilon,\eta}(t)\rightarrow P_t$ such that $X_{\epsilon,\eta}(t)$ is the mapping cylinder of $\psi_t$. 

Moreover, there exists also a continuous simplicial map $r_t:X_{\epsilon,\eta}(t)\rightarrow X_{\epsilon,\eta}(0)$ sending $P_t$ onto $p$ and inducing a stratified homeomorphism from $X_{\epsilon,\eta}(t)\setminus P_t$ onto $X_{\epsilon,\eta}(0)\setminus\{p\}$.
\end{theorem}
\begin{remark}
We point out a few differences between L\^e's theorem and the construction of gradient-like vector fields in this paper.
The flow $\Psi_t$ over a line in $D_{\eta_2}$ in \cite{L2} is constructed away from
the collapsing polyhedron  using a gradient-like vector field . We show how to extend such a flow to  that of vanishing polyhedron along a line in the next section. 
L\^e also discussed an extension of the theorem on a semidisc but with flows converging to the singular point $p$. 
This increases the dimension of the unstable set and is not appropriate for our purpose. Moreover, on $(f^{c})^{-1}(0)$, the vector field is not defined. In fact, a possibly natural extension may provides identically vanishing vector field on $(f^{c})^{-1}(0)$, which is not a gradient-like vector field. Later, we show how to overcome these difficulties using the additional Morse condition while keeping the unstable set as the collapsing polyhedron 
as in L\^e's theorem.
\end{remark}

The construction in \cite{L2} is given by induction on the complex dimension of $X$ with the following hypothesis:
\begin{hypothesis}\label{ind}
For all $t\in D_\eta\setminus\{0\}$, there exist in $X_{\epsilon,\eta}(t)$
\begin{enumerate}
\item a polyhedron $P_t$ (called vanishing polyhedron), adapted to the stratification $\cal S$ and whose real dimension is $dim_\C X -1$.
\item a vector field $E_t$ defined on all $X_{\epsilon,\eta}(t)$, tangent to each stratum of\\
$\stackrel{\circ}{X}_{\epsilon,\eta}(t)\setminus P_t$ induced by ${\cal S}$, with $\stackrel{\circ}{X}_{\epsilon,\eta}(t) =X_{\epsilon,\eta}(t)
\setminus \partial X_{\epsilon,\eta}(t)$,
that is continuous, weakly Lipschitz, and non-zero on $\stackrel{\circ}{X}_{\epsilon,\eta}(t)\setminus P_t$ , transverse to strata of $\partial X_{\epsilon,\eta}(t)$ (and toward interior), and zero on $P_t$. Also we suppose that the flow given by the vector field defines a continuous, surjective, simplicial map of $\partial X_{\epsilon,\eta}(t)$
onto $P_t$ and that $ X_{\epsilon,\eta}(t)$ is the mapping cylinder of  such map.
\end{enumerate}
\end{hypothesis}
In the case when $dim_\C X =1$, it is easy to see that the induction hypothesis holds true.
First, we explain the case of $dim_\C X =2$ in more detail closely following \cite{L2} and explain
the modifications for general inductive construction.

\subsection{The case of dimension two}
For simplicity, suppose that $f^c(p)=0$ and $p \in \overline{S}_i$ for all $S_i \in \cal S$.
From the topological preparation theorem (\cite{L2}), we can replace $X_{\epsilon,\eta}(t)$ by
$X_t = X_{\epsilon,\eta_1,\eta_2} \cap (f^c)^{-1}(t)$.
The subset $X_t$ is naturally stratified and we denote by ${\cal S}(t)$ its stratification. 

For convenience, we set $$D_t:=D_{\eta_1} \times \{t\}$$
and also denote by  $\phi_t: X_t \to D_t$ the map induced by the morphism $\phi$ (or the linear form $l$)
which is surjective.
 
By the definition of Cerf diagram, the restriction of $\phi_t$ to the strata of ${\cal S}(t)$ is of maximal rank at all the points of
 $\stackrel{\circ}{X}_t \setminus \cup \phi^{-1}(\Delta_\alpha)$, where $\stackrel{\circ}{X}_t = X_t \setminus \partial X_t$. The number of intersections between the Cerf diagram $\Delta$ and
$D_t$ is finite and is denoted as $y_i(t) \in D_t \cap (\cup \Delta_\alpha)$.
Note also that we have a topological covering
of $X_t \setminus \phi_t^{-1}(\Delta_\alpha)$ on $D_t \setminus(\cup \Delta_\alpha)$.

We remark that here it exists also a case where the polar curve does not exist. In \cite{L2},
this trivial case has not been discussed, but in such a case,
as the map $\phi$ is submersive away from $p$, the proof can be carried out without much difficulty.

Fix a point $\lambda_t \in D_t \setminus (\cup \Delta_\alpha)$.
Consider  simple paths from $\lambda_t$ contained in the interior of $D_t$ whose end points are $y_i(t)$ so that $\lambda_t$ is the unique common point of these paths (see figure 1).
We call $\delta(y_i(t))$ these paths between $\lambda_t$ and $y_i(t)$. Denote by $Q_t$ the union of such
paths.

\begin{figure}\begin{center}
\includegraphics[height=2.0 in]{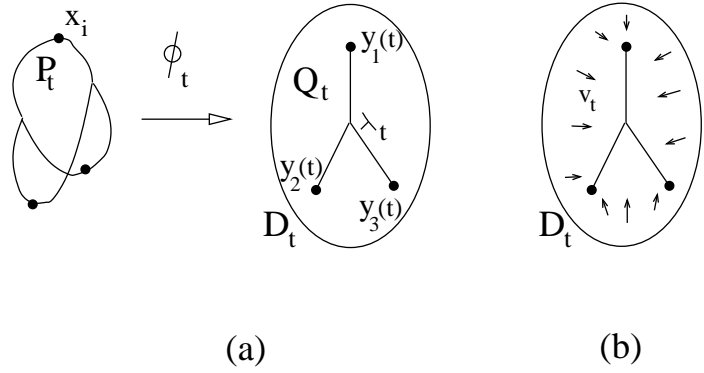}
\caption{(a) $Q_t$ and $P_t$ in dimension two, (b) Vector field $v_t$}
\label{fig1}
\end{center}
\end{figure}

We can also construct in $D_t$ a $C^\infty$ vector field which we call $v_t$ and which vanishes on $Q_t$,
tranverse to $\partial D_t$, toward interior, integrable. We may also assume that its flow $p_t:[0,\infty) \times (D_t - Q_t) \to D_t$
defines a map $\xi_t:\partial D_t \to Q_t$, by defining $\xi_t(u) = \lim_{\tau \to \infty}p_t(\tau,u)$ for
all $u \in \partial D_t$, that is surjective and continuous, simplicial and differentiable in the interior of
each simplices. 
As $\phi_t$ is a covering above $D_t \setminus Q_t$, we can lift a vector field $v_t$  to the vector field $E_t$, differentiable on $X_t \setminus \phi_t^{-1}(Q_t)$, vanishing on $\phi_t^{-1}(Q_t)$, continuous on $X_t$, integrable and transverse to $\partial X_t$ and toward interior.

Also, if $q_t:[0,\infty) \times (X_t - \phi_t^{-1}Q_t) \to X_t$ is the flow associated to $E_t$, 
we have a continuous and surjective map $\psi_t:\partial X_t \to \phi_t^{-1}(Q_t)$
 that is simplicial and differentiable on the interior of each simplex, and
 $X_t$ is the mapping cylinder of $\psi_t$.
In the case of dimension two, the polyhedron  $P_t$ is defined as $\phi_t^{-1}(Q_t)$ which is naturally stratified by the stratification induced by ${\cal S}(t)$.

\subsection{The collapsing in dimension two}
Parametrized version of the construction in the previous subsection has been carried out in \cite{L2} in two ways. One is for $t$ along a curve $t \in \gamma$ converging to $p$. The other is over a semi-disc containing such $\gamma$. The first case is explained here to define a collapsing cone, and for the second case, we refer to the general case of the induction.

Consider $D_{\eta_2}$ and a simple path $\gamma$ which joins $0$ and $t_0 \in \partial D_{\eta_2}$ and which is
transverse to $\partial D_{\eta_2}$.
We can make the construction of the vector field $E_t$ simultaneously for all $t$ along $\gamma(t)$.
We can choose the common point $\lambda_t$ to be for example, the barycenter of the points $y_i(t) \in D_t \cap (\cup \Delta_\alpha)$.
The projection onto $D_{\eta_2}$ induces a covering of $\Delta$ onto $D_{\eta_2}$ that is only ramified at $0$.
Therefore, the inverse image of $\gamma\setminus\{0\}$ defines $k$ curves on $\Delta$ each of which is diffeomorphic to
$\gamma\setminus\{0\}$ by the projection onto $D_{\eta_2}$. Such simple curve has $0$ in its closure and connects to
each $y_i(t)$. In the same way, $\lambda_t$ gives another distinct simple path $\Lambda$ whose
projection to $D_{\eta_2}$ induces a diffeomorphism to $\gamma$ outside $0$.
We can also choose the path $\delta(y_i(t))$ such that
$T_i = \cup_{t \in \gamma} \delta(y_i(t))$ forms differentiable triangle in $\cup_t D_t$ outside $0$.
The triangles $T_i$ have  in common only the simple path $\Lambda$ given by $\lambda_t$.
We can construct the vector fields $v$ continuous on $\cup_{t \in \gamma} D_t = D_{\eta_1} \times \gamma$,
differentiable on $\cup_{t \in \gamma} D_t  \setminus \cup T_i$, vanishing on $Q:=\cup T_i$, transverse
to $\partial D_{\eta_1} \times \gamma$, whose projection to $\gamma$ is zero and inducing a flow
$$p:[0,\infty) \times (D_{\eta_1} \times \gamma\setminus\cup T_i) \to D_{\eta_1}\times \gamma$$
which defines a map $\xi:\partial D_{\eta_1} \times \gamma \to Q$
by $\xi(z) = \lim_{\tau \to \infty} p(\tau,z)$ for all $z \in \partial D_{\eta_1} \times \gamma$, that is
continuous, surjective, simplicial, and differentiable in the interior of each simplex
(see figure (\ref{fig2}): this is a simplified figure and the triangles $T_i$ may not be linear as in the figure as
the points $\{ y_i(t) \}$ changes along $\gamma(t)$).
\begin{figure}\begin{center}
\includegraphics[height=2.2in]{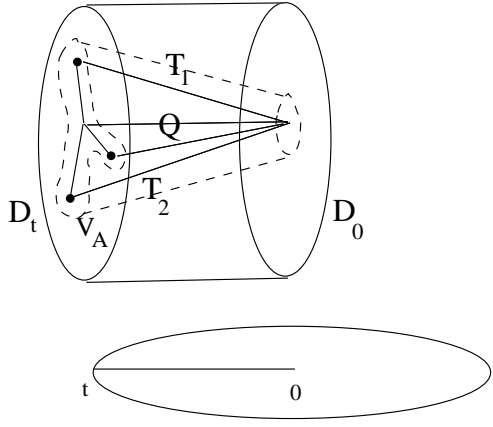}
\caption{$Q$ and its contracting  neighborhoods $V_A$}
\label{fig2}
\end{center}
\end{figure}

We denote for all $A >0$ by $V_A(Q)$ the closed neighborhood of $Q$ defined by
$$V_A(Q):= D_{\eta_1} \times \gamma \setminus p ( [0,A) \times \partial D_{\eta_1} \times \gamma )$$
where $\partial V_A(Q)$ is a differentiable manifold which has a projection onto $\gamma$
with a circle as fiber.

 As $\phi$ has maximal rank away from the Cerf diagram $\Delta$, the
space $\phi^{-1}(\partial V_A(Q)) \cap B_\epsilon(x)$ is a differentiable submanifold of $X_{\epsilon,\eta_1,\eta_2}$
that is a proper local trivial fibration on $\gamma$. 

Set $P := \phi^{-1}(Q) \cap B_\epsilon(p)$.
We construct a vector field on $X_{\epsilon,\eta_1,\eta_2} \cap f^{-1}(\gamma)\setminus P$ in the
following way. We fix some $A>0$.
\begin{enumerate}
\item If $y \notin \phi^{-1}(V_A(Q)) \cap B_\epsilon(p)$, there exists an open neighborhood $U_y$ of $y$
in $D_{\eta_1} \times \gamma$ that does not meet the closed set $\phi^{-1}(V_A(Q)) \cap B_\epsilon(p)$.
We define in $U_y$ a $C^\infty$ vector fields $G_y$ that lifts a differentiable vector field $\theta$
which is not zero on $\gamma$ and which takes $t_0$ to $0$ in a finite time $a >0$.
\item If $y \in \phi^{-1}(V_A(Q)) \cap B_\epsilon(p)$, there exists an open neighborhood $U_y$ of $y$ in
$D_\eta \times \gamma$ that does not intersect $P$ and in it, one can construct a vector field $G_y$
that lifts a vector field $\theta$, defined in $(1)$, and that is tangent to $\phi^{-1}(\partial V_{A'}(Q))$ for
all $A' >A$.
\end{enumerate}
We cover thus $X_{\epsilon,\eta_1,\eta_2} \cap f^{-1}(\gamma)\setminus P$  by the open sets $U_y$ with the vector fields $G_y$ in each of them
and by considering a partition of unity $g_y$ subordinate to the covering $U_y$, we define
the vector field $G = \sum g_y G_y$.
Such a vector field is differentiable in $X_{\epsilon,\eta_1,\eta_2} \cap f^{-1}(\gamma)\setminus P$,
 lifts $\theta$ and is tangent to $\phi^{-1}(\partial V_{A'}(Q))$ for $A' >A$.

The flow $\pi:[0,a] \times Z \to Z$ defined on $Z=X_{\epsilon,\eta_1,\eta_2}\setminus P$
gives a $C^\infty$ diffeomorphism of $X_{t_0}\setminus P_{t_0}$ onto $X_0\setminus \{0\}$ which extends to a continuous mapping of $X_{t_0}$ to $X_0$ sending $P_{t_0}$ to $\{0\}$.
$P$ is called the collapsing cone above $\gamma$.

L\^e also constructed collapsing cone over the semi-disc $D^-$ along the angular rays as described in (\ref{eqfoli}), which he used in the inductive construction(\cite{L2}).

\subsection{Inductive construction of polyhedron}
The general inductive construction which is  explained in great detail in \cite{L2} is quite involved and we only
give a brief scketch of some parts of them. 

Suppose that $dim_\C(X) \geq 3$ and that the induction hypothesis is verified in all dimensions of
the situation of dimension $\leq n-1$. To prove the induction statement for $X$ with $dim_\C (X) = n$, L\^e used the induction hypothesis at two different parts of $X$ as follows:
First consider
$$x_i(t) = \phi^{-1}(y_i(t)) \cap (\cup \Gamma_\alpha),$$
which is an isolated critical point over $y_i(t)\in D_t \cap (\cup \Delta_\alpha)$.
Consider also $\lambda(t)$ and simple paths $\delta(y_i(t))$ in $D_t$ as in the case of dimension two.
We further suppose that  $\lambda(t) = (0,t) \in D_{\eta_1} \times D_{\eta_2}$ because $\{0\} \times \C$ is not a component of $\Delta$ when $l \in \Omega$ (otherwise, this violates the isolated singularity condition).

Firstly, we  restrict $X$ to a hyperplane in $\Omega$ defined by $\{l=0\}$. The restriction $X \cap \{l=0\}$ is of dimension $(n-1)$ and we may apply the induction hypothesis for the restricted function $f^c|_{ X \cap (l=0)} \to \C$. The fiber $X_t \cap (l=0)$ have a vanishing polyhedron $P_t'$ and a vector field that we call $E_t'$. 

On the other hand, by taking a sufficiently small ball $D_s(y_i(t))$ of the critical value $y_i(t)$, we consider the projection 
$$\phi_t:X_t \cap \phi^{-1}(D_s(y_i(t))) \to D_s(y_i(t)),$$
which has also an isolated singularity at $x_i(t)$. As we have  $dim_\C X_t = n-1$, we can again apply the induction
hypothesis and we obtain a vanishing polyhedron $P_i(a_i)$ at the point $a_i \in \delta(y_i(t)) \cap \partial D_s(y_i(t))$,
and also a collapsing cone of $P_i(a)$ along $\delta(y_i(t)) \cap D_s(y_i(t))$, which we
denote by $\WT{P}_i$.

The above two polyhedra $P_t'$ and $\WT{P}_i$ are to be connected along the path $\delta(y_i(t))$ from $a_i$ to $\lambda(t)$. This is done by flowing the polyhedra $P_i(a)$ along a vector field $W'$ to $P_t'$ where it
will be glued to.
 
We set $\delta' = \delta(y_i(t))\setminus  D_s(y_i(t))$. 
Note that above $\delta_i'$ , the morphism $\phi_t$
induces a stratified fibration which is locally trivial. In fact, over the $ \delta_i'$, L\^e considers a controlled vector field
 $\Xi$ that is tangent to each stratum, weakly Lipschitz and lifts a $C^\infty$ vector fields $\xi$ which is a  non-zero vector field on $\delta_i'$ and whose flow goes from $a_i$ to $0$.
 
Then, the vector fields $E_t'$ of $\phi_t^{-1}(0)$ is  tranported onto all the fibers of $\phi_t^{-1}(u)$  for all $u \in \delta_i'$
via the flow $\Xi$, and we obtain the vector field $W$ on $ \phi_t^{-1}( \delta_i')$ whose restriction to $\phi_t^{-1}(0)$ is the field $E_t'$
and the restriction to $\phi_t^{-1}(u)$ is a vector field on  $\phi_t^{-1}(u)$.

 Consider a differentiable function $h$ on $\delta_i'$ that takes value 0 at 0 and which is non-zero and positive  on $\delta_i' \setminus \{0\}$. We regard $h$ as a function on $A_i =\phi_t^{-1}(\delta_i')$.  On $A_i$, we have a field $W' = W + h \Xi$ which is weakly Lipschitz, tangent to each stratum of $\stackrel{\circ}{A_i}$ and which lifts $h \xi$. 

The image of $P_i(a_i)$ under the flow of $W'$ is also a subpolyhedron $P_i'$ of $P_t'$. The trajectories of points of $P_i(a_i)$ by $V_1$ yield a polyhedron $R_i$. 

We define thus a polyhedron $S_i = \tilde{P}_i \cup R_i \cup P_i'$  and the vanishing polyhedron $P_t$ is defined by
$$P_t = P_t' \cup_{i=1}^k S_i.$$

The vector field $E_t$, which  is the contracting vector field to the vanishing polyhedron, can be constructed, and we
refer readers to \cite{L2} for its construction.
\subsection{Collapsing over semi-disc}
In this subsection, we briefly recall the parametrized construction of the above over the semi-disk $D^- =
\{z \in D^2_{\eta_2}| Re(z)<0\}$ exposed in \cite{L2}.

Recall that the collapsing cone $Q$ along $\gamma_0$
constructed in section 5.2 is of real dimension two. Similarly on $D^-$, we first construct the three dimensional $Q$ as follows.
Consider in $D_{\eta_1} \times D^-$ 
the Cerf diagrams $\cup \Delta_\alpha$, which projects
differentiably onto $D^-$ away from $0$. Denote its intersection points with each fiber over $t \in D^-$ as $y_i(t)$ for $i=1,\cdots,k$. For each $t \in D^-$,  take $\lambda(t)$ to be the barycenter of critical values $y_i(t)$ (we assume that this does not lie on the Cerf diagram), and choose paths $\delta(y_i(t))$ from $\lambda(t)$ to $y_i(t)$ continously on $t$.
Then, for each $i$, we obtain a set $$T_{i}=\cup_{t \in D^-} \delta(y_i(t)).$$ 
The $T_{i}$'s have in common the intersection $\Lambda=\cup_{t \in D^-}\{\lambda(t)\}$.
We set $Q:= \cup_{i} T_i$, which is a three dimensional figure with ribs given by $T_i$'s. 

The collapsing polyhedron of \cite{L2} is constructed in the following way.
In (5.3.2) of \cite{L2}, for each critical point $x_i(t)$ of $\phi$ above the critical value $y_i(t)$ with $t \in D^-$, L\^e chooses a radius $r(t)>0$ such that
$\cup_{t \in D^-} B_{r(t)}(x_i(t))=:B_i$ provides a sharp neighborhood
of $\cup_{t \in D^-} \{ x_i(t)\}$.  Here, sharp means that the neighborhood shrinks to $p$ or the radius converges to $0$ as $t$ approaches $0$. We may assume that each $B_i$ is disjoint and their closures meet at $p$.
One chooses such $r(t)$ as a real analytic function of $t
\in D^-$. 
For this, consider (as in (5.3.3) of \cite{L2}) 
 a real analytic function $s(t)$ with $1 \gg r(t) \gg s(t) >0$ for $t \in D^-$ and
define in $\C^2$ the set
$${\cal D}_i = \cup_{t \in D^-} D_{s(t)}(y_i(t)).$$
Also, consider a sharp (and sufficiently small) neighborhood $U$ of $Q \setminus \{0\}$ which meets ${\cal D}_i$ for $i=1,\cdots k$ but
does not intersect $y_i(t)$. As $U$ does not meet the Cerf diagram, L\^e uses the trivialization of $V = X_{\epsilon,\eta_1,\eta_2} \cap \phi^{-1}(U)$
over $D^-$. 
In this setting, L\^e may repeat the inductive construction of the previous subsection to 
define the collapsing polyhedron for all $t \in D^-$ simultaneously. The collapsing
vector field $E$ is also analogously constructed. We refer readers to \cite{L2} for more details.

\section{A proof of the main theorem}
Let $p$ be a critical point of $f$. After restricting to a normal slice through $p$, we can suppose the set $\{p\}$ to form a 0-dimensional stratum. We assume for simplicity that $f(p)=0$. Here is the slightly stronger form of the main theorem \ref{thm:main}.
\begin{theorem}\label{thm:normal}
With the above hypothesis, there exists a stratified  weakly controlled gradient-like vector field $V$ near $p$ for $f$ with continuous flow and whose unstable and stable set $W^u(p)$ and $W^s(p)$ at $p$ satisfy
 for every $S \in \cal S$, 
\begin{eqnarray*}
\dim_{\R} (W^u(p) \cap S) &\leq& \dim_{\C} S  \\
\dim_{\R} (W^s(p) \cap S) &\leq& \dim_{\C} S. 
\end{eqnarray*}
Moreover they satisfy 
$$\dim_{\R} W^u(p)=\dim_{\R} W^s(p)=n$$
\end{theorem}

\begin{corollary}
Including the tangential direction, the above theorem \ref{thm:normal} implies the corollary \ref{cor:main}.
\end{corollary}
\begin{proof}
For the case $n=1$, the theorem can be proved without much difficulty and we first explain this case, that is $\dim_{\C}X =1$, before
proving the general case.
As we consider only the neighborhood of $p$, we may assume that $X$ is given the stratification
$$\mathcal S=\{S_0=\{p\},S_1= X \setminus \{p \} \}.$$
Recall that  $df^c(p) \neq 0$ since the function $f$ is Morse and $\{p\}$ is a stratum of $\cal S$. Consider
$X_{\epsilon, \eta_2} = X \cap B_\epsilon (p) \cap (f^c)^{-1}(D_{\eta_2})$ and the map $f^c:X_{\epsilon,\eta_2} \to D_{\eta_2}$. As  $X_{\epsilon,\eta_2}$ is
of complex dimension one, and $df^c \neq 0$, $f^c$ defines a holomorphic map which is a local diffeomorphism
except at $p$.  

We consider the following vector field $V_{st}$ on $D_{\eta_2}$. If the coordinate of $D_{\eta_2}$ is
given by $x+ y\sqrt{-1}$, we define
\begin{equation}\label{stvf}
V_{st}(x,y)=\sqrt{x^2+y^2}\frac{\partial}{\partial x}.
\end{equation}

One can easily check that induced flows are parallel to the $x$-axis and, particularly, that the only gradient lines which can have the origin as limit point are those contained in the $x$-axis. Also one can check that
the flow of the vector field takes infinite time to approach the point $0$ if one starts at negative real axis.

Now, we consider a map $f^c:X_{\epsilon,\eta_2} \setminus \{p\} \to D_{\eta_2} \setminus \{0 \}$, which is submersive.
Hence, we can take a controlled lift $V'$ of the vector field $V_{st}$ by the map $f^c$ with respect to
some control system on $X$, i.e. we have
$f^c_* (V') = V_{st}$. The lifted vector field $V'$ may be considered a continuous vector field on $X$ by
defining $V'(p)=0$. (In fact $X \setminus \{p\}$ near $p$ may considred as one stratum, any lift is a controlled lift.)

To make sure that the vector field $V$ is weakly controlled at $p$, let $\rho$ be the
distance function in the chosen control system of $X$. 
The first condition $d\Pi_S\circ V'_{|T_S\cap R}=V'\circ\Pi_S$ is trivially satisfied as $\Pi_S$ is the projection onto $S=\{p\}$. To meet the condition $|d\rho_S\circ V'_{|T_S}|\leq A\rho_S(x)$ for some $A$,
we may multiply a smooth function $h(\rho)$
of $\rho$ vanishing at $p$ to the constructed vector field $V'$ to  obtain
new vector field $V:=h(\rho) \cdot V'$. By a suitable choice of $h(\rho)$ the second condition
also can be satisfied, while the trajectories of $V'$ equal the trajectories of $V$ as sets.
It is easy to see that the stable and unstable set of the vector field $V$ on $X$ are real one dimensional
sets which are given by several half-lines with common vertex $p$.
%
Hence, this proves the theorem for the case $n=1$.

Now, we consider the general case $n \geq 2$. The proof consists of the following steps.
The first step is to consider L\^e's construction of the collapsing polyhedron $P$ and the contracting vector field $E$  over a semi-disc $D^- = \{z \in D_{\eta_2}| Re(z) <0 \}$, and use it to lift the vector field (\ref{stvf}). This
vector field on $D^-$ is different from that of L\^e, it will provide a
gradient-like flow over $X \cap B_\epsilon(x) \cap (f^c)^{-1}(D^{-})$ whose unstable set of the critical point will be given by the collapsing polyhedron over the negative real axis.
Here to show that such a flow is well-defined along the collapsing polyhedron, we need to modify
the construction of L\^e, to find a special lifting of  the vector field in $D^-$ in a neighborhood of $P$.

The second step is to consider an analogous construction over a semi-disc $D^+$, and glue them along the intersection $D^+ \cap D^-$.
The final step is to have the additional construction of a vector field in a neighborhood $N_0$ of $(f^c)^{-1}(0) \setminus \{0\}$, then the final vector field will be obtained by partition of unity and will prove the main theorem.

The first step of  constructing such a vector field over a semi-disc $D^-$ can be  divided into the following procedures which will be explained later in detail.
\begin{enumerate}\label{firstst}
\item Following \cite{L2}, we consider parametrized constructions of a vanishing polyhedron $P_t$ in $f^{-1}(t)$ and a contractible vector field $E_t$ for all $t \in D^-$, and denote the resulting collapsing polyhedron
as $P$ and the contracting vector field as $E$. By integrating the flow of $E$, we define a closed neighborhood $V_A(P)$ of $P$ as in \cite{L2}.
\item Consider a sharp neighborhood of $P$ which is defined as $V_{1/|f^c(x)|}(P)$ over $D^-$, and
consider a map $f^c: V_{1/|f^c(x)|}(P) \to \C$ which is a submersion. 
\item Using the trivialization obtained by applying Thom-Mather isotopy lemma to $f^c$, we redefine
the collapsing polyhedron $\WT{P}$, and contracting vector field $\WT{E}$ (and denote it again by $P$ and $E$ for simplicity)
\item Consider the standard vector field (\ref{stvf}) on $D^-$, and define a specific lifting to $X_{\epsilon,\eta_1,\eta_2} \cap f^{-1}(D^-)$ which is tangent to $V_A(P)$ away from $P$,
so that the lifted vector field extends continously on $P$ and can be integrated.
\end{enumerate}

For simplicity we use the following notation
$$X(D^-):=X_{\epsilon,\eta_1,\eta_2} \cap (f^c)^{-1}(D^-)$$

Following \cite{L2} the construction of collapsing cone over the semi-disk, we have a 
collapsing polyhedron $P$ and a contracting vector field $E$ on $X(D^-)$, which is continuous, weakly Lipschitz,
vanishing at $P$ and its associated flow provides a continuous strata preserving map
$$\rho:[0,\infty) \times X_{D^-} \to X_{D^-}$$
which sends $X_{D^-} \setminus P$ to itself.
We denote by $\partial_{out} X_{D^-}$ the following part of the boundary of $X(D^-)$:
$$X\cap B_\epsilon \cap \phi^{-1}(\partial D_{\eta_1} \times D^-), X\cap \partial B_\epsilon \cap \phi^{-1}(D_{\eta_1} \times D^-)$$
We remark that the vector field $E$ is transverse to $\partial_{out} X_{D^-}$
on which it points inward, moreover its projection under $f^c$ vanishes.
We may define the closed neighborhood of $P$ by
$$ V_A(P):= X(D^-) \setminus \rho \big( [0,A) \times \partial X_{out}(D^-) \big).$$
It is not hard to see that $f^c:V_A(P) \to D^-$ is still a submersion
from the Morse condition of $f$ in the stratified settting. The main purpose of the contracting vector field $E$ is to construct the above closed neighborhood $V_A(P)$,
so that one can define a flow preserving these levels determined by $V_A(P)$.

L\^e then considers (as in section 5.2) a lifting the vector field on $D^-$ whose flow converges to $0$ and whose
trajectories provides a  foliation of $D^-$ given by
the lines $\gamma_\theta$ where 
\begin{equation}\label{eqfoli}
\gamma_\theta = \{ \eta_2 te^{i \theta}| t \in [0,1] \} \;\;\;\; \textrm{for}  \; \frac{\pi}{2} < \theta < \frac{3\pi}{2}.
\end{equation}
(See figure (\ref{fig3}) (a)). Such a lifting is used to define a mapping sending $X(D^-) - P$ to 
$X_0 \setminus \{p\}$

\begin{figure}\begin{center}
\includegraphics[height=2.0in]{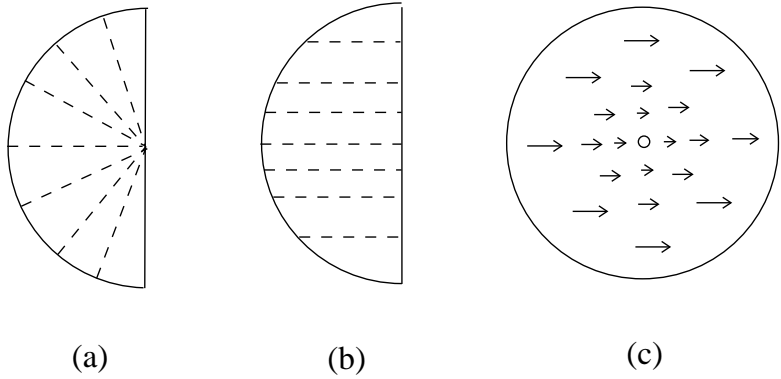}
\caption{(a) Flows in \cite{L2} (b) Flows in this paper (c) $V_{st}$}
\label{fig3}
\end{center}
\end{figure}

 Instead, we make a similar construction over the semi-disc $D^-$, but considering the standard vector field $V_{st}$ of (\ref{stvf}) whose trajectory gives the foliation of $D^-$ by the horizontal lines (figure (\ref{fig3})(b))
$$\gamma_j = \{ z \in D^{-} | Im (z) = j \} \;\; \textrm{for} \; -\eta_2 < j < \eta_2.$$
We remark that the Cerf diagram does not intersect $\{0\} \times D_{\eta_2}$, hence none of the $\{0\} \times \gamma_j$'s.
This finishes procedure (1).

There is a subtle point in lifting the vector field.  L\^e's construction of the lifting as in  section 5.2 is actually
carried out away from the collapsing polyhedron. Hence L\^e obtains a map from $X(D^-) \setminus P$ to $X_0 \setminus \{p\}$ via a flow
which, however, is not shown to extend over $P$, as the lifted vector field may not have a limit toward the collapsing polyhedron $P$.
(Intuitively, it seems clear that such a vector field should extend
over $P$ but it turns out to be rather subtle). 

In our case, we would like to define the vector field everywhere, hence we proceed in a slightly different way to prove the
extension. We now explain procedure (2). Instead of $V_A(P)$, we consider a sharp neighorhood by changing $A$ to $\infty$ as
we approach the point $p$. By sharp, we mean that the neighborhood shrinks to the point $p$ as we approach it.
 Namely, we may set $A = \frac{1}{|f^c(x)|}$ for each $x \in X(D^-)$ or define
$$ V_{sharp}(P):= X(D^-) \setminus \big(\cup_{x\in X_{out}(D^-)}\rho ([0,\frac{1}{|f^c(x)|}) \times \{x\})\big).$$
As $f^c:V_{sharp}(P) \to D^-$ is still a proper submersion, by the Thom-Mather isotopy lemma, we obtain a stratified trivialization $$ \Psi: V_{sharp}(P) \cong D^- \times F_t$$ via
$f^c$, where $F_t = V_{sharp}(P) \cap (f^c)^{-1}(t)$ for the given $t \in \partial D^-$, so that the
projection of $\Psi$ to the first component equals the map $f^c$. We may regard polar curves as additional stratas and that this trivialization preserves these stratras also.

Given this trivialization, we will construct a new collapsing polyhedron denoted by $\WT{P}$ and a new contracting vector field denoted by $\WT{E}$, using $P_t$ and $E_t$ for $t=-\eta_2$. This is somewhat unsatisfactory  but otherwise we do not
know how to construct a gradient-like vector field which extends to $P$.

Via the trivialization $\Psi$
we consider parallel transports of  the vanishing polyhedron $P_t$ and the contracting vector field $E_t|_{F_t}$ over $D^-$ and define the collapsing polyhedron $\WT{P}$ and the contracting vector field $\WT{E}$ in $D^- \times F_t$, and hence in $V_{sharp}(P)$.
Note that parallel transport sends a point of the polar curve to another point of the polar curve.
This provides a new collapsing polyhedron $\WT{P}$, which a priori may be different from $P$, but over a contractible set $D^-$ these should contain equivalent informations. Also, we may extend $\WT{E}$ from $V_{sharp}(P)$ to
$X(D^-)$ so that it is continuous, weakly Lipschitz, transverse to $\partial X_{out}(D^-)$, pointing inward and 
whose flow provides a retraction onto $\WT{P}$. We may also assume that with the new vector field $\WT{E}$,
the time it takes for the flow from $\partial X_{out}(D^-)$ to $V_{sharp}(P)$ is a smooth function
on $D^- \times F_t \subset D^- \times M$. This is procedure (3). We now fix $\WT{P},\WT{E}$ and call it again by $P$ and $E$ for simplicity of
expressions.

Now, we explain how to choose a lifting of the standard vector field (\ref{stvf}) on $D^-$. We need to
do it carefully to make sure that the flow  extends  continuously over $P$. 

We construct a lifting of $V_{st}$ in $D^-$ to 
$X(D^-) \setminus P$ in the
following way. We choose $A_0>2$.
\begin{enumerate}
\item If $y \notin V_{A_0}(P)$ for $y \in X_{D^-}$, there exists an open neighborhood $U_y$ of $y$ in $X(D^-)$
  that does not meet the closed set $V_{A_0}(P)$.
In $U_y$, we define a $C^\infty$ vector field $\WT{G}_y$ as a lifting of the differentiable vector field 
$V_{st}$.

\item If $y \in V_{A_0}(P) \setminus V_{sharp}(P)$, there exists an open neighborhood $U_y$ of $y$ in
$X(D^-) \setminus V_{sharp}'(P)$ where $V_{sharp}'(P)$ is constructed in the same way as $V_{sharp}(P)$ with
$A = \frac{2}{|f^c(x)|}$ instead. 
Hence $U_y$ does not intersect $P$ and in it, one can construct a vector field $\WT{G}_y$
that lifts a vector field $V_{st}$ and that is tangent to $\partial V_{A'}(P)$ for
all $A' >A_0$. The lifting exists as $\phi$ is submersive away from the polar curve.

\item If $y \in V_{sharp}(P) \setminus P$, we find a special lifting in the following way. Recall that we have
a trivialization $\Psi: V_{sharp}(P) \cong D^- \times F_t$. Denote $\Psi(y) = (f^c(y),b(y))$.
Consider the curve $\sigma:[-c,c] \mapsto D^- \times F_t$ defined by translation, i.e. for $s\in [-c,c]$,
$$\sigma(s) =(f^c(y) + s,b(y)).$$
By assumption, there exists a smooth function $A(s)$ such that each point $\Psi^{-1}(\sigma(s))$ lies in the time $A(s)$ level $V_{A(s)}(P)$
of the flow $\rho$ for a smooth function $A(s)$. In fact $A(s)$ will not be constant, as the flow $\rho$ starts from
$\partial X_{out}(D^-)$ and not from $\partial V_{sharp}(P)$.

But by choosing smaller $c'$, with $0<c'<c$ if necessary, we may assume that
$$\WT{\sigma}(s):=\rho_{A(0) - A(s)}(\Psi^{-1} \circ \sigma(s))$$ lies in $\WT{V}$ for $ -c' < s<c'$, and
the new curve $\WT{\sigma}(s)$ will be tangent to $V_{A(0)}(P)$ by the construction.
Consider the image curve $\WT{\sigma}(-c',c')$ and we take a unique vector $\WT{G}(y)$ tangent to $\WT{\sigma}$
at $s=0$ which satisfies $(df^c)\WT{G}(y)=V_{st}$.
In this way, we define a unique lifting $\WT{G}$ of $V_{st}$ in $V_{sharp}(P) \setminus P$.
\end{enumerate}

The main reason for the last construction is to use the trivialization $\Psi$ as a reference.
Namely, as the new contracting vector field $\WT{E}$ is defined in $V_{sharp}(P)$ via parallel tranport of the  trivialization,
the map $\rho_{A(0) - A(t)}$ in fact commutes with the trivialization.  Hence, we have
$$\Psi(\WT{\sigma}(s)) = \big(f^c(y)+s, \rho_{A(0) - A(s)} (b(y)) \big)$$

Now, we claim that the lifting $\WT{G}$ extends to $P$ continuously. Intuitively, this is because we have chosen a special lifting so that there exists no ambiguity in its limit approaching $P$. We will
show that the fiber component of the lifted vector field $\WT{G}(y)$ in terms of the trivialization of $\Psi$ vanishes
in a uniform way as we appraoch $P$.

We denote  $\Psi_*(\WT{G}(y)) = (V_{st}(f^c(y)),\WT{G}_F(y))$.
As $|V_{st}(z)| = |z|$, it is easy to see that the fiber component is
$$|\WT{G}_F(y)| = |\frac{d}{ds}|_{s=0} \rho_{A(0) - A(|f^c(y)|s)}(b(y))| = |f^c(y)|\cdot |E(y)|\cdot |\frac{dA}{ds}(0)|.$$
This is because the flow $\rho$ is generated by the vector field $E$ or $E_t$.

As $y \in V_{sharp}(P) \setminus P$ approach the collapsing polyhedron $P$ ( or as $A(0)$ approaches $\infty$), the $|E(y)|$ approaches zero, whereas $|f^c(y)|$ and $|\frac{dA}{ds}(0)|$ is bounded in its local compact neighborhood and hence the fiber component vanishes as we approach $P$.  This proves the desired property of the lifted vector field near $P$, and we extend $|\WT{G}(y)|$ continuously as $(V_{st},0)$ in the trivialization $D^- \times F_t$.

It is not hard to see that the resulting vector field is weakly Lipschitz, and its associated flow is stratum preserving. Hence via $\Psi$, we can extend the lifted vector field to $P$ in $V_{sharp}(P)$.

We cover thus $X(D^-)\setminus P$  by the open sets $U_y$ with the vector fields $\WT{G}_y$
and by considering a partition of unity $g_y$ subordinate to the covering $U_y$, we define
the vector field $$\WT{G} = \sum g_y \WT{G}_y.$$
Such a vector field is stratified and differentiable in $X(D^-)\setminus P$,
lifts $V_{st}$, is tangent to $\partial V_{A'}(P)$ for $A' >A$, and preserves $P$.

The flow of the vector field $\WT{G}$ gives a diffeomorphism from $X_{t_0}\setminus P_{t_0}$ to $X_0\setminus\{p\}$ 
for $t_0 \in \gamma_0$ and  from $X_{t_j}\setminus P_{t_j}$ to $X_{t^{'}_j}\setminus P_{t^{'}_j}$ for $t_j,t_j' \in \gamma_j$ and it extends to a continuous map from $X_{t_0}$ to $X_0$ and from $X_{t_j}$ to $X_{t^{'}_j}$, sending respectively $P_{t_0}$ to $\{p\}$ and $P_{t_j}$ to $P_{t^{'}_j}$.

Hence $\WT{G}$ provides over $X(D^-)$ the desired  stratified gradient-like vector field, which has real $n$-dimensional stable set given by $\cup_{t_0 \in \gamma_0} P_{t_0}$. This finishes the first step of the proof.

We claim that the polyhedron $\cup_{t \in \gamma_0} P_t$ as an unstable set of $p$ satisfies the dimension estimate (\ref{dimest1}) of the main theorem . This can be checked by considering carefully L\^e's inductive procedure (section 5.3) constructing $P_t$.
The estimate is clearly true for $n=1$. We remind that the vanishing polyhedron is constructed in section 5.3 by gluing  the polyhedron $R_i$, obtained by flowing (via $W'$) the vanishing polyhedron $P_i(a)$
to $P_t'$, where  by induction we may assume  that $P_i(a)$ and $ P_t'$ satisfy the estimates (\ref{dimest1}).
Here the flow $W'$ is   stratum-preserving as it is the sum of a controlled vector field ($h \Xi$) and a weakly Lipschitz vector field $W$. As a consequence, the dimensional estimate (\ref{dimest1}) continues to hold throughout the construction of $P_t$. 

Now, we begin the second step of the proof.
As the argument above used only the contractibility of $D^-$, we can in fact perform the
same construction for a slightly larger set
$$\widetilde{D}^- := \{z\in D_{\eta_2} \setminus \{0\}\; |\; \pi - \epsilon  < arg(z) < 3\pi + \epsilon \},$$ 
where $arg(z)$ is defined appropriately, and 
the vector field, which we call $\WT{G}^-$, over
$\widetilde{D}^- $ gives rise to a stable set at the critical point $p$ of the expected real dimension.

Also, we carry out a similar construction over the set
$$\widetilde{D}^+ := \{z\in D_{\eta_2} \setminus \{0\}\; |\; -\pi - \epsilon < arg(z) < \pi + \epsilon \},$$
with the vector field induced from $V_{st}$ on it. Hence, the vector field, which we call $\WT{G}^+$, over
$\widetilde{D}^+ $ gives rise to an unstable set of the expected real dimension at the critical point $p$.
Note that $\WT{D}^+ \cup \WT{D}^-  = D_{\eta_2} \setminus \{0\}$, and also $\WT{D}^+ \cap \WT{D}^-$ does not
contain $\gamma_0$. 

Now, we glue the two fields $\WT{G}^\pm$ by means a partition of unity to obtain a stratified weakly Lipschitz and controlled gradient-like vector field 
$V'$ over  $D_{\eta_2} \setminus \{0\}$. 
We can use a partition of unity depending only on the angular parameter of $D_{\eta_2}$.
Note that both the vector fields $\WT{G}^\pm$ approaches zero as we approach 
the point $p$. Hence, the gluing is well-defined near $(f^c)^{-1}(0)$. But the vector field $V'$, if we are to extend it continously over $X$, must vanish on $(f^c)^{-1}(0)$, hence $V'$ can not be the gradient-like vector field over $X$ we are looking for
(as a gradient-like vector field  is required to vanish only at critical points).

Still, the new vector field $V'$ projects down to $V_{st}$ by the map $f^c$, and hence the stable and unstable set of $p$ exist only along $\gamma_0$. So the stable (resp. unstable) set of $\WT{G}^+$ (resp. $\WT{G}^-$) remains as
the stable (resp. unstable) set of the vector field $V'$ at $p$.

Up to now we have a vector field $V'$ defined in $B_\epsilon(p)\setminus (f^c)^{-1}(0)$. It may or may not
extend continuously over $B_\epsilon(p)$. Hence, to solve this problem, 
we consider in addition, another vector field $V_0$ defined in the following way.

Consider a neighbourhood $N_0 \subset (B_\epsilon \setminus \{p\})$ of $(f^c)^{-1}(0)\setminus \{p\} $. 
Recall from the theorem \ref{h1} that $p$ is an isolated point in $\Gamma_{\overline{S_i}} \cap (f^c)^{-1}(0)$, and hence
$N_0$ can be chosen away from polar curves.

As we consider gradient-like vector fields, it is clear that their unstable and stable set of $p$ do not intersect with
$(f^c)^{-1}(0)$ except at $p$. Hence as they are closed sets, we may  also assume that the
neighborhood $N_0$ does not intersect both the unstable and stable sets of $p$ (see figure \ref{fig4})
and does not intersect the polar curves.

\begin{figure}\begin{center}
\includegraphics[height=2.3in]{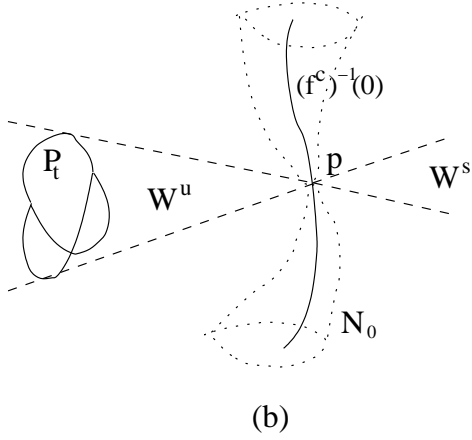}
\caption{Neighborhood $N_0$}
\label{fig4}
\end{center}
\end{figure}

As a consequence, $\phi|_{N_0}$ is submersive. So we can take a controlled lift, with respect to some control system, of the vector field 
$$V_{st}'=\frac{\partial}{\partial x_3}$$ 
on $D_{\eta_1}\times D_{\eta_2} \subset \C^2$, where $(x_1 + x_2 \sqrt{-1},x_3 +x_4\sqrt{-1})$ denotes the standard coordinates in $\C^2$. The lift, which we denote by $V_0'$, is a stratified, controlled, gradient-like (because of the compatibility with $\phi$) vector field.

Note that the flow of $V_0'$ preserves the hyperplanes $H+c$ (as it preserves the linear form $l$), and does not vanish
in $N_0$. To make it continuous  at $p$, we  actually consider the vector field 
$$V_0:=\rho V_0',$$ where $\rho$ is a distance function from $p$ in a fixed control system of $X$ and $V_0$ is still stratified, controlled (except at $p$) and gradient-like.

The open sets $\{ (X_{\epsilon,\eta_1,\eta_2} \setminus (f^c)^{-1}(0)), N_0\}$ define an open cover of $B_\epsilon(p) \setminus \{p\}$ and,
by using a partition of unity subordinate to this cover, we glue  the vector field $V'$ with $V_0$, and
denote the resulting vector field by $V$.
Hence we get the  continuous vector field $V$ on $X_{\epsilon,\eta_1,\eta_2}$ by defining $V(p)=0$ 

We claim that the vector $V$ is the desired gradient-like vector field for $f$ of the main theorem.

Indeed, first note that, as $V'$ and $V_0$ are both stratified gradient-like vector fields, 
so is the vector field $V$. It is also easy to check that $V$ does not vanish except at $p$. 
It remains to show that the unstable and stable manifolds of $V$ satisfies
the desired properties.

We now consider the unstable and stable set of $p$ for the vector field $V$.
As we have glued away from the stable and unstable set at $p$ of the vector field $V'$, these
remain as subsets of the stable and unstable set at $p$ of the vector field $V$.
Hence it remains to show that in the intersection of two open sets
$$(X_{\epsilon,\eta_1,\eta_2} \setminus (f^c)^{-1}(0)) \cap N_0, $$
we do not create any new gradient line converging to or emanating from $p$ by gluing.

To this purpose, we recall first that both vectors $f^c_*(V')$ and $f^c_*(V_0)$  are parallel to the
$x$-axis of $D_{\eta_2}$, hence the glued vector field $V$ also has such a property.
For a possibly new flow trajectory converging to or emanating from $p$, we only need to
check what happens over $\gamma_0$, which is the $x$-axis. The vector field $V'$ over $\gamma_0$ is, by construction,
the same as the vector fields $\WT{G}^\pm$ on $D^\pm$. 

Now, recall that the vector field $V'$ inside the set $V_A(P) \setminus P$ is a lift of $V_{st}$ for
the map $f^c$ and is tangent to $\partial V_{A'}(P)$ for some $A'$, hence
the flow of $V'$ preserves  the sets $\partial V_{A'}(P)$ for $A'>A$.

Consider the polar curve $\Gamma$ and a linear form $l:\Gamma \to D_{\eta_1}$. This is a branch cover at $p$, hence, as
$t \in \gamma_0$ approaches $0$, the image of the polar curve $l(x_i(t)) = y_i(t)$ also approaches $0$, thgerefore, we may assume that
in a sufficiently close neighborhood of $p$ and for sufficiently large $A'$, the flow along  $\partial V_{A'}(P)$ decreases the value of $|l|$.

Note however that the flow $V_0$ at a point of $\partial V_{A'}(P)$ fixes $|l|$, as it preserves the
linear form, hence, the vector field $V_0$ at $\partial V_{A'}(P)$ is pointing outward (away from $P$), toward $\partial V_{A''}(P)$ for
$A'' <A'$. 
As $V'$ is tangent to $\partial V_{A'}(P)$ and as $V_0$ is pointing outward to $\partial V_{A'}(P)$,
the final vector field $V$, which is obtained as a partition of unity of $V'$ and $V$, is also tangent to  $\partial V_{A'}(P)$ or pointing outward.

This implies that the flow of $V$ does not approach $p$ in the intersection of these open sets, since
$A'$ should go to $\infty$ for the flow to approach $p$.

The vector field $V$ is controlled except at $p$. We can make it weakly controlled while preserving the
trajectories of $V$, as in the case of $n=1$, by multiplying it with a function of $\rho$ which vanishes at $0$.
(here $\rho$ is the distance function to $p$ from the fixed control system).

As the unstable and stable set are given by $P|_{(f^c)^{-1}(\gamma_0)\cap P}$, and we have
already shown that they satisfy the dimensional estimate (\ref{dimest1}) after the
construction of $P$.
This finishes the proof of the claim and this proves the main theorem.
\end{proof}
Now, we prove corollary \ref{cor:main}.
\begin{proof}
To obtain a global vector field on $X$, we first apply the main theorem to each normal slice of
critical points of $f$. Then, by adding a tangential component (as in the smooth case) to each normal slice,
we obtain a vector field with the required properties in each neighborhood of critical points.
For any point $q \in X$ which is not a critical point, we choose a contractible neighorhood $U$ of $q$ which
does not contain any critical point. Then we consider a map $f:U \to \R$, which is a submersion,
and take a controlled lift of the standard unit vector field on $\R$.
Then, we obtain the global vector field by a partition of unity.
\end{proof}

As L\^e also shows in \cite{L2} we have the following corollary:

\begin{corollary}
\label{mc}
There exists a map $\chi:\partial X_t\cap B_\epsilon(p)\rightarrow W^u(V)\cap X_t\cap B_\epsilon(p)$ such that $X_t\cap B(p)$ is the mapping cylinder of $\chi$.
\end{corollary}
\begin{proof}
Observe that by construction $W^u(V)\cap X_t\cap B_\epsilon(p)=P_t$, so we can take as $\chi$ the flow of $E_t$ for $t\rightarrow+\infty$. 
\end{proof}
This corollary allows us to relate Morse data to the unstable set of the vector field $V$ as follows:
\begin{corollary}
The homotopy type of normal Morse data is given by 
$$(W^u(V),W^u(V)\cap B(p))$$ 
(or by $(W^s(V),W^s(V)\cap B_\epsilon(p))$).
\end{corollary}

\begin{proof}
For a complex analytic variety the homotopy type of normal Morse data is given by $(Cone({\cal L}),{\cal L})$, where ${\cal L}$ denotes the complex link of $p$ and $Cone({\cal L})$ the cone over ${\cal L}$; but by definition ${\cal L}$ is $X_t\cap B_\epsilon(p)$, which, by corollary \ref{mc}, deformation retracts onto $W^u(V)\cap B_\epsilon(p)$; on the other hand, by construction, $W^u(p)$ is the cone over $P_t=W^u(V)\cap B(p)$.
\end{proof}
 
\begin{corollary}
The normal Morse data are homeomorphic to $(J,K)$ where
$$J=Cone([(Cyl(\chi)\times[0,2\pi])/(q,0)\cong(\mu(q),2\pi)]\cup_{\partial(X_t\cap B(p))\times S^1}[\partial(X_t\cap B(p))\times D^1])$$
$$K=Cyl(\chi)\times[0,2\pi]$$
where $\mu$ is the monodromy associated to $p$, $\chi:\partial X_t\cap B(p)\rightarrow W^u(V)\cap X_t\cap B(p)$ is the map appearing in corollary \ref{mc} and $K\subset J$ is embedded in the base of the cone.
\end{corollary}

\begin{proof}
This follows \textit{mutatis mutandis} from the homeomorphism type of normal Morse data in \cite{GMP}.
\end{proof}

\end{document}